\documentclass{amsart}
\usepackage{graphicx}

\newtheorem{thm}{Theorem}[section]

\newtheorem{lem}[thm]{Lemma}

\theoremstyle{definition}

\theoremstyle{remark}
\newtheorem{rem}[thm]{Remark}
\numberwithin{equation}{section}
\DeclareMathOperator{\RE}{Re}

\newcommand{\set}[1]{\left\{#1\right\}}

\begin{document}

\title[Sufficient Conditions for Starlike  Functions]{Sufficient Conditions for Starlike  Functions Associated with the Lemniscate of Bernoulli}
\author{S. Sivaprasad Kumar}
\address{Department of Applied Mathematics, Delhi Technological University, Delhi-110042, India}
\email{spkumar@dce.ac.in}
\author{Virendra Kumar}
\address{Department of Applied Mathematics, Delhi Technological University, Delhi-110042, India}
\email{vktmaths@yahoo.in}
\author{ V. Ravichandran}
\address{Department of Mathematics, University of Delhi, Delhi---110007, India}
\email{vravi@maths.du.ac.in}

\author{N. E. Cho}
\address{Department of Applied Mathematics, Pukyong National University, Busan 608-737}
\email{necho@pknu.ac.kr}


 \date{}
 \dedicatory{Dedicated to Prof. Hari M. Srivastava}
\begin{abstract}Let $-1\leq B<A\leq 1$.  Condition on $\beta$, is determined so that $1+\beta zp'(z)/p^k(z)\prec(1+Az)/(1+Bz)\;(-1<k\leq3)$ implies $p(z)\prec \sqrt{1+z}$. Similarly, condition on $\beta$ is determined  so that $1+\beta zp'(z)/p^n(z)$ or $p(z)+\beta zp'(z)/p^n(z)\prec\sqrt{1+z}\;(n=0, 1, 2)$ implies $p(z)\prec(1+Az)/(1+Bz)$ or $\sqrt{1+z}$. In addition to that condition on $\beta$ is derived so that $p(z)\prec(1+Az)/(1+Bz)$ when $p(z)+\beta zp'(z)/p(z)\prec\sqrt{1+z}.$ Few more problems of the similar flavor are also considered.
\end{abstract}

\maketitle

\section{Introduction}
Let $\mathcal{A}$ be the class of analytic functions defined on the unit disk $\mathbb{D}:=\{z\in\mathbb{C}:|z|<1\}$ normalized by the condition $f(0)=0=f'(0)-1.$ For two analytic functions $f$ and $g$, we say that $f$ is
\emph{subordinate} to $g$ or $g$ \emph{superordinate} to $f$, denoted by $f\prec g$, if
there is a Schwarz function $w$ with $|w(z)|\leq |z|$ such that
$f(z)=g(w(z))$. If $g$ is univalent, then $f\prec g$ if and only if
$f(0)=g(0)$ and $f(\mathbb{D})\subseteq g(\mathbb{D})$. For an analytic function $\varphi$ whose range is starlike with respect to $\varphi(0)=1$ and is symmetric with respect to the real axis, let
$\mathcal{S}^*(\varphi)$ denote the class of \emph{Ma-Minda starlike functions} consisting of all $f\in\mathcal{A}$ satisfying $zf'(z)/f(z)\prec \varphi(z)$.
For special choices of $\varphi$, $\mathcal{S}^*(\varphi)$  reduce to well-known subclasses of starlike functions. For example, when   $-1\leq B< A\leq 1$, $\mathcal{S}^*[A,B]:=\mathcal{S}^*((1+Az)/(1+Bz))$ is the Janowski starlike functions \cite{jano} (see \cite{yasar}) and $\mathcal{S}^*[1-2\alpha, -1]$ is the class  $\mathcal{S}^*(\alpha)$ of  starlike
functions of order $\alpha$ and  $\mathcal {S}^*:=\mathcal{ S}^*(0)$ is  the class of starlike functions.
For $\varphi(z):=\sqrt{1+z}$, the class $\mathcal{S}^*(\varphi)$ reduces to the class  $\mathcal{SL}$ introduced by
Sok\'o\l\ and Stankiewicz \cite{sokol96} and studied recently by Ali\ et al. \cite{lemi, lemi1}.
A function $f\in\mathcal{A}$ is in the class $\mathcal{SL}$ if $zf'(z)/f(z)$ lies in the region bounded by the right half plane
of the lemniscate of Bernoulli given by $|w^2-1|<1$. Analytically,
$\mathcal{SL}:=\set{f\in\mathcal{A}:\left|\left(zf'(z)/f(z)\right)^2-1\right|<1}.$
For $b\geq1/2$ and $a\geq1$, a more general class $S^*[a,b]$ of the functions $f$ satisfying $|(zf'(z)/f(z))^a-b|<b$ was considered by Paprocki and Sok\'o\l\ \cite{sokol}. Clearly $S^*[2,1]=:\mathcal{SL}.$
For some radius problems related with lemniscate of Bernoulli see \cite{lemi2,lemi1,sokol09,sokol96}. Estimates for the initial coefficients of functions in the class $\mathcal{SL}$ is available in \cite{sokol09}.

Let $p$ be an analytic function defined on $\mathbb{D}$ with $p(0)=1$. Recently Ali et al. \cite{lemi} determined conditions for $p(z)\prec \sqrt{1+z}$ when  $1+\beta zp'(z)/p^k(z) $ with $k=0,1,2$ or $(1-\beta)p(z)+\beta p^2(z)+\beta zp'(z)$ is subordinated to $\sqrt{1+z}$.
Motivated by the works in \cite{lemi,lemi2, lemi1,sokol,sokol09}, in Section~\ref{2} condition on $\beta$ is determined so that  $p(z)\prec \sqrt{1+z}$ when
$1+\beta zp'(z)/p^k(z)\prec(1+Az)/(1+Bz),\; (-1<k\leq3)$. Similarly, condition  on $\beta$ is determined so that $p(z)\prec(1+Az)/(1+Bz)$ when $1+\beta zp'(z)/p^n(z)\prec\sqrt{1+z}$, n=0, 1, 2. Further condition on $\beta$ is obtained in each case so that $p(z)\prec\sqrt{1+z}$ when $p(z)+\beta zp'(z)/p^n(z),\;n=0, 1, 2.$ At the end of this section problem $p(z)+\beta zp'(z)/p(z)\prec\sqrt{1+z}$ implies $p(z)\prec (1+Az)/(1+Bz)$ is also considered.

 Silveramn \cite{silver} introduced the class $\mathcal{G}_b$ by
  $$\mathcal{G}_b:=\set{f\in \mathcal{A}:\left|\frac{zf''(z)/f'(z)}{zf'(z)/f(z)}-1\right|<b}$$ and proved $\mathcal{G}_b\subset\mathcal{S}^*(2/(1+\sqrt{1+8b}), 0<b\leq1$. Further this result was improved by Obradovi$\check{c}$ and Tuneski \cite{obra} by showing $\mathcal{G}_b\subset S^*[0, b]\subset\mathcal{S}^*(2/(1+\sqrt{1+8b}), 0<b\leq1.$ Tuneski \cite{tuneski1} further obtained the condition for $\mathcal{G}_b\subset S^*[A, B].$ Inspired by the work of Silverman \cite{silver}, Nunokawa et al. \cite{nuno}, obtained the, sufficient conditions for function in the class $\mathcal{G}_b$ to be strongly starlike, strongly convex, or starlike in $\mathbb{D}$.  By setting $p(z)=zf'(z)/f(z)$, the inclusion $\mathcal{G}_b\subset S^*[A, B]$ can be written as $$1+\frac{zp'(z)}{p^2(z)}\prec 1+b z\Longrightarrow p(z)\prec \frac{1+Az}{1+Bz}.$$ Recently Ali et al.\cite{ali}, obtained condition on the constants $A, B, D, E\in [-1, 1]$ and $\beta$ so that $p(z)\prec(1+Az)/(1+Bz)$ when $1+\beta zp'(z)/p^n(z)\prec (1+Dz)/(1+Ez),\; n=0, 1.$ In Section~\ref{3}, alternate and easy proof of results \cite[Lemma 2.1, 2.10]{ali} are discussed. Further this section has been concluded with condition on $A, B, D, E\in [-1, 1]$ and $\beta$ such that $1+\beta zp'(z)/p^2(z)\prec (1+Dz)/(1+Ez)$ implies $p(z)\prec(1+Az)/(1+Bz).$

The following results are required in order to prove our main results:
\begin{lem}\cite[Corollary 3.4h, p.135]{miller}\label{miller} Let $q$  be univalent in $\mathbb{D}$, and let $\varphi$ be analytic in a domain $D$ containing $q(\mathbb{D})$. Let $zq'(z)\varphi(q(z))$ be starlike. If $p$ is analytic in $\mathbb{D},$ $p(0)=q(0)$ and satisfies  $$zp'(z)\varphi(p(z))\prec zq'(z)\varphi(q(z)),$$ then $p\prec q$ and $q$ is the best dominant.
\end{lem}
The following is a more general form of the above lemma:
\begin{lem} \cite[Corollary 3.4i, p.134]{miller}\label{miller2} Let $q$  be univalent in $\mathbb{D}$, and $\varphi$ and $\nu$ be analytic in a domain $D$ containing $q(\mathbb{D})$ with $\varphi(w)\neq0$ when $w\in q(\mathbb{D})$. Set
\[Q(z):=zq'(z)\varphi(q(z)),\quad h(z):=\nu(q(z))+Q(z).\]
Suppose that
\begin{enumerate}
\item  $h$ is convex or $Q(z)$ is starlike univalent in $\mathbb{D}$
 and \item
${\RE } \left(\frac{zh'(z)}{Q(z)}\right)>0$ for  $z\in \mathbb{D}.$
\end{enumerate}
If
 \begin{equation} \label{e1}
    \nu(p(z))+zp'(z)\varphi(p(z))\prec
    \nu(q(z))+zq'(z)\varphi(q(z)),
\end{equation} then $p\prec q$ and $q$ is the best dominant.
\end{lem}
\begin{lem}\cite[Corollary 3.4a, p.120]{miller}\label{miller1} Let $q$ be analytic in $\mathbb{D}$, let $\phi$ be analytic in a domain $D$ containing $q(\mathbb{D})$ and suppose
\begin{enumerate}
  \item $\RE \phi[q(z)]>0$
  and either
  \item  $q$ is convex, or
  \item $Q(z)=zq'(z).\phi[q(z)]$ is starlike.
\end{enumerate}
If $p $ is analytic in $\mathbb{D}$, with $p(0)=q(0),$ $p(\mathbb{D})\subset D$ and
\[ p(z)+zp'(z)\phi[p(z)]\prec q(z),\]
then $p(z)\prec q(z).$
\end{lem}
\section{Results Associated with Lemniscate of Bernoulli}\label{2}
In the first result condition on $\beta$ is obtained so that the subordination $$1+\beta \frac{zp'(z)}{p^k(z)}\prec\frac{1+Az}{1+Bz}\;\; (-1<B<A\leq1)$$ implies $p(z)\prec\sqrt{1+z}.$
\begin{lem}\label{lem4}
  Let $|\beta|\geq 2^{(k+3)/2}(A-B)+|B\beta|, -1< k\leq3.$ Let $p$ be an analytic function defined on $\mathbb{D}$ with $p(0)=1$ satisfies
  $$1+\beta \frac{zp'(z)}{p^k(z)}\prec\frac{1+Az}{1+Bz}\quad  (-1<B<A\leq1),$$ then $p(z)\prec\sqrt{1+z}$
\end{lem}
\begin{proof}
 Let $q(z)=\sqrt{1+z}$. A computation shows that the function $$Q(z):=\beta \frac{zq'(z)}{q^k(z)}=\frac{\beta z}{2(1+z)^{(k+1)/2}} \;\; (-1<k\leq 3)$$ is starlike in the unit disk $\mathbb{D}$. Consider the subordination $$1+\beta \frac{zp'(z)}{p^k(z)}\prec 1+\beta \frac{zq'(z)}{q^k(z)}.$$ Thus in view of Lemma \ref{miller}, it follows that $p(z)\prec q(z)$.  In order to prove our result, we need to prove
  $$\frac{1+Az}{1+Bz}\prec 1+\frac{\beta z q'(z)}{q^k(z)}=1+\frac{\beta z}{2(1+z)^{(k+1)/2}}:=h(z)$$
   Let $w=\Phi(z)=\frac{1+Az}{1+Bz}.$ Then $\Phi^{-1}(w)=\frac{w-1}{A-Bw}.$ The subordination $\Phi(z)\prec h(z)$ is equivalent to $z\prec \Phi^{-1}(h(z))$. Thus in order to prove result, we need only to show $|\Phi^{-1}(h(e^{it}))|\geq1$.
  For $z=e^{it}, -\pi \leq t\leq \pi$, we have $$|\Phi^{-1}(h(e^{it}))|\geq\frac{|\beta|}{2(A-B)(2\cos (t/2)) ^{(k+1)/2}+|B\beta|}=:g(t).$$ A calculation shows that $g(t)$ attains its minimum at $t=0$. Further the value of $g(t)$ at $\pi$ or $-\pi$ comes out to be $1/|B|$ which is naturally greater than the value at the extreme point $t=0$ because if $g(0)\geq g(\pi)$, then $(A-B)|\beta|\leq 0$ which is absurd. Thus $$g(0)=\frac{|\beta|}{2^{(k+3)/2}(A-B)+|B\beta|}\geq1$$ for $|\beta|\geq 2^{(k+3)/2}(A-B)+|B\beta|$.
  Hence $\Phi(z)\prec h(z)$ and the proof is complete now.
\end{proof}
Next result depicts the condition on $\beta$ such that $1+\beta zp'(z)\prec \sqrt{1+z}$
 implies $p(z)\prec (1+Az)/(1+Bz)\;(-1\leq B<A\leq1).$
 On subsequent lemmas similar results are obtained by considering the expressions $1+\beta zp'(z)/p(z)$ and $1+\beta zp'(z)/p^2(z)$.
\begin{lem}\label{lm5}
  Let $(A-B)\beta\geq\sqrt{2}(1+|B|)^2+(1-B)^2$ and $-1\leq B<A\leq1.$ Let $p$ be an analytic function defined on $\mathbb{D}$ with $p(0)=1$ satisfies
  $$1+\beta zp'(z)\prec\sqrt{1+z},$$ then $p(z)\prec\frac{1+Az}{1+Bz}.$
\end{lem}
\begin{proof}
  Define the function $q:\mathbb{D}\rightarrow\mathbb{C}$ by $$q(z)=\frac{1+Az}{1+Bz}\;\;\;(-1\leq B<A\leq1)$$ with $q(0)=1$. A computation shows that $$Q(z)=\beta zq'(z)=\frac{\beta(A-B)z}{(1+Bz)^2}$$ and $$\frac{zQ'(z)}{Q(z)}=\frac{1-Bz}{1+Bz}.$$ Let $z=re^{it}, r\in(0, 1), -\pi\leq t\leq \pi$. Then
  \begin{eqnarray*}
    \RE\left(\frac{1-Bz}{1+Bz}\right) &=& \RE\left(\frac{1-Bre^{it}}{1+Bre^{it}}\right)\\
    &=&  \frac{1-B^2r^2}{|1+Bre^{it}|^2}.
  \end{eqnarray*}
  Since $1-B^2r^2>0\; (|B|\leq 1, 0<r<1)$ and so $\RE (zQ'(z)/Q(z))>0$, this shows that $Q$ is starlike in $\mathbb{D}.$  It follows from Lemma \ref{miller}, that the subordination  $$1+\beta zp'(z)\prec1+\beta zq'(z)$$ implies $p(z)\prec q(z).$ Now we need to prove the following in order to prove lemma:
    $$\sqrt{1+z}\prec1+\beta zq'(z)=1+\beta\frac{(A-B)z}{(1+Bz)^2}=:h(z).$$
  Let $w=\Phi(z)=\sqrt{1+z}.$ Then $\Phi^{-1}(w)=w^2-1.$
   The subordination $\Phi(z)\prec h(z)$ is equivalent to the subordination
 $z\prec \Phi^{-1}(h(z)).$ Now in order to prove result it is enough to show $|\Phi^{-1}(h(e^{it}))|\geq1,\; z=e^{it},\;-\pi\leq t\leq \pi.$ Now
 $$|\Phi^{-1}(h(e^{it}))|=\left|\left(1+\beta\frac{(A-B)e^{it}}{(1+Be^{it})^2}\right)^2-1\right|\geq1\;\;\text{implies that}\;\;\left|1+\beta\frac{(A-B)e^{it}}{(1+Be^{it})^2}\right|\geq\sqrt{2}.$$
 Further \begin{eqnarray*}
     \left|1+\beta\frac{(A-B)e^{it}}{(1+Be^{it})^2}\right|&=& \frac{|1+(2 B+\beta(A-B))e^{it}+B^2e^{2it}|}{|1+2Be^{it}+B^2e^{2it}|}\\
     &\geq& \frac{\RE(2 B+\beta(A-B)+B^2e^{it}+e^{-it})}{1+2|B|+B^2}\\
     &=& \frac{2 B+\beta(A-B)+(1+B^2) \cos t}{(1+|B|)^2}\\
     &\geq& \frac{2 B+\beta(A-B)-(1+B^2)}{(1+|B|)^2}\geq\sqrt{2}
  \end{eqnarray*}
  for $(A-B)\beta\geq\sqrt{2}(1+|B|)^2+(1-B)^2$. Therefore $\Phi(z)\prec h(z)$ and this completes the proof.
\end{proof}
\begin{lem}
  Let $(A-B)\beta\geq(\sqrt{2}-1)(1+|A|)(1+|B|)$ and $-1\leq B<A\leq1.$ Let $p$ be an analytic function defined on $\mathbb{D}$ with $p(0)=1$ satisfies
  $$1+\beta \frac{zp'(z)}{p(z)}\prec\sqrt{1+z},$$ then $p(z)\prec\frac{1+Az}{1+Bz}.$
\end{lem}
\begin{proof}
  Let the function $q:\mathbb{D}\rightarrow\mathbb{C}$ be defined by $$q(z)=\frac{1+Az}{1+Bz}\;\;\;(-1\leq B<A\leq1).$$ A computation shows that $$Q(z):=\frac{\beta zq'(z)}{q(z)}=\frac{\beta(A-B)z}{(1+Az)(1+Bz)}$$ and $$\frac{zQ'(z)}{Q(z)}=\frac{1-ABz^2}{(1+Az)(1+Bz)}.$$ Let $z=re^{it}, r\in(0, 1), -\pi\leq t\leq \pi$. Then
  \begin{eqnarray*}
    \RE\left(\frac{1-ABz^2}{(1+Az)(1+Bz)}\right) &=& \RE\left(\frac{1-ABr^2e^{2it}}{(1+Are^{it})(1+Bre^{it})}\right)\\
    &=&  \frac{(1-ABr^2)(1+(A+B)r \cos t+ABr^2)}{|1+Are^{it}|^2|1+Bre^{it}|^2}.
  \end{eqnarray*}
  Since $1+ABr^2+(A+B)r\cos t\geq (1-Ar)(1-Br)>0$ for $A+B\geq0$ and similarly $1+ABr^2+(A+B)r\cos t\geq (1+Ar)(1+Br)>0$ for $A+B\leq0$. It follows that $Q$ is starlike in $\mathbb{D}.$ Lemma \ref{miller} suggests us that the subordination  $$1+\beta \frac{zp'(z)}{p(z)}\prec1+\beta \frac{zq'(z)}{q(z)}$$ implies $p(z)\prec q(z).$ Now we have to prove
  $$\sqrt{1+z}\prec1+\beta \frac{zq'(z)}{q(z)}=1+\frac{\beta(A-B)z}{(1+Az)(1+Bz)}=:h(z).$$
  Let $w=\Phi(z)=\sqrt{1+z}.$ Then $\Phi^{-1}(w)=w^2-1.$
 The subordination $\Phi(z)\prec h(z)$ is equivalent to the subordination
 $z\prec \Phi^{-1}(h(z)).$  Now in order to prove result it is enough to show $|\Phi^{-1}(h(e^{it}))|\geq1,\; -\pi\leq t\leq \pi.$  Now
  $$\left|\Phi^{-1}(h(e^{it}))\right|=\left|\left(1+\frac{\beta(A-B)e^{it}}
  {(1+Ae^{it})(1+Be^{it})}\right)^2-1\right|\geq1\;\;\text{implies that}\;\;\left|1+\frac{\beta(A-B)e^{it}}
  {(1+Ae^{it})(1+Be^{it})}\right|\geq\sqrt{2}.$$ Further
  \begin{eqnarray*}
    \left|1+\frac{\beta(A-B)e^{it}}
  {(1+Ae^{it})(1+Be^{it})}\right| &\geq& \RE\left(1+\frac{\beta(A-B)e^{it}}{(1+Ae^{it})(1+Be^{it})}\right) \\
   &=& 1+\frac{(A-B)\beta}{(1+|A|)(1+|B|)}\geq\sqrt{2},\end{eqnarray*}
  for $(A-B)\beta\geq(\sqrt{2}-1)(1+|A|)(1+|B|)$. Therefore $\Phi(z)\prec h(z)$ and this completes the proof.
\end{proof}
\begin{lem}
  Let $(A-B)\beta\geq(\sqrt{2}-1)(1+|A|)^2+(1-A)^2$ and $-1\leq B<A\leq1.$ Let $p$ be an analytic function defined on $\mathbb{D}$ with $p(0)=1$ satisfies
  $$1+\beta \frac{zp'(z)}{p^2(z)}\prec\sqrt{1+z},$$ then $p(z)\prec\frac{1+Az}{1+Bz}.$
\end{lem}
\begin{proof}
  Let the function $q:\mathbb{D}\rightarrow\mathbb{C}$ be defined by $$q(z)=\frac{1+Az}{1+Bz}\;\;\;(-1\leq B<A\leq1)$$ with $q(0)=1$. Then $$Q(z)=\frac{\beta zq'(z)}{q^2(z)}=\frac{\beta(A-B)z}{(1+Az)^2}$$ and $$\frac{zQ'(z)}{Q(z)}=\frac{1-Az}{1+Az}.$$ Let $z=re^{it}, -\pi\leq t\leq \pi, 0<r<1$. Then
$$ \RE\left(\frac{1-Az}{1+Az}\right)=\frac{1-A^2r^2}{|1+Are^{it}|^2}.$$
  Since $1-A^2r^2>0\; (|A|\leq 1, 0<r<1)$. Hence $\RE (zQ'(z))/Q(z)>0$, this shows that $Q$ is starlike in $\mathbb{D}.$ An application of Lemma \ref{miller} reveals that the subordination  $$1+\beta \frac{zp'(z)}{p^2(z)}\prec1+\beta \frac{zq'(z)}{q^2(z)}$$ implies $p(z)\prec q(z).$ Now our result established if we prove
  $$\sqrt{1+z}\prec1+\beta \frac{zq'(z)}{q^2(z)}=1+\beta\frac{(A-B)z}{(1+Az)^2}=:h(z).$$
 Rest part of the proof is similar to that of Lemma~\ref{lm5} and therefore it is skipped here.
 \end{proof}
 In the next result condition on $\beta$ is obtained so that $p(z)+\beta zp'(z)\prec \sqrt{1+z}$ implies $p(z)\prec \sqrt{1+z}$. On subsequent lemmas similar results are discussed by considering the expressions $p(z)+\beta zp'(z)/p(z)$ and $p(z)+\beta zp'(z)/p^2(z)$.
\begin{lem} Let $p$ be an analytic function defined on $\mathbb{D}$ with $p(0)=1$ satisfies
$p(z)+\beta zp'(z)\prec \sqrt{1+z},\; \beta>0$. Then $p(z)\prec \sqrt{1+z}$.
\end{lem}
\begin{proof}Define the function $q:\mathbb{D}\rightarrow \mathbb{C}$ by
$q(z)=\sqrt{1+z}$ with $q(0)=1$. Since  $q(\mathbb{D})=\{w :
|w^2-1|<1\}$ is the right half of the lemniscate of Bernoulli,
$q(\mathbb{D})$ is a convex set and hence $q$ is a convex function.
Let us define $\phi(w)=\beta$ then \[ \RE \phi[q(z)]= \beta>0.\]
Consider the function $Q$ defined by
$$Q(z):= zq'(z)\phi(q(z))=\beta\frac{z}{2\sqrt{1+z}}.$$ Further
\begin{eqnarray*}
  \RE\left(\frac{zQ'(z)}{Q(z)}\right) &=& 1-\RE\left(\frac{z}{2(1+z)}\right) \\
   &\geq& \frac{3}{4}>0.
\end{eqnarray*}
 Thus the function $Q$ is starlike and the result now follows by an application of Lemma~\ref{miller1}.
\end{proof}
\begin{lem} Let $p$ be an analytic function defined on $\mathbb{D}$ with $p(0)=1$ satisfies
$$p(z)+\beta\frac{zp'(z)}{p(z)}\prec \sqrt{1+z}, \; \beta>0.$$ Then $p(z)\prec \sqrt{1+z}$.
\end{lem}
\begin{proof} As before let $q$ be given by
$q(z)=\sqrt{1+z}$ with $q(0)=1$. Then $q$ is a convex function.
Let us define $\phi(w)=\beta/w$. Since  $q(\mathbb{D})=\{w :
|w^2-1|<1\}$ is the right half of the lemniscate of Bernoulli and so  \[ \RE \phi[q(z)]=\frac{\beta}{|\sqrt{1+z}|^2}\RE \left( \sqrt{1+z}\right)>0.\]
Consider the function $Q$ defined by
$$Q(z):= \beta\frac{zq'(z)}{q(z)}=\beta\frac{z}{2(1+z)}.$$ Further
\begin{eqnarray*}
  \RE\left(\frac{zQ'(z)}{Q(z)}\right) &=& 1-\RE\left(\frac{z}{1+z}\right) \\
   &\geq& \frac{1}{2}>0.
\end{eqnarray*} Thus the function $Q$ is starlike and the result now follows by an application of Lemma~\ref{miller1}.
\end{proof}
\begin{lem} Let $p$ be an analytic function defined on $\mathbb{D}$ with $p(0)=1$ satisfies
$$p(z)+\beta\frac {zp'(z)}{p^2(z)}\prec \sqrt{1+z},\;\;\beta>0.$$ Then $p(z)\prec \sqrt{1+z}$.
\end{lem}
\begin{proof} Let $q$ be given by
$q(z)=\sqrt{1+z}$ with $q(0)=1$. Then $q$ is a convex function.
Let us define $\phi(w)=\beta/w^2$ and  \[ \RE \phi[q(z)]=\RE \left(\frac{\beta}{1+z}\right)>\frac{\beta}{2}>0.\]
Consider the function $Q$ defined by
$$Q(z):= \beta\frac{zq'(z)}{q^2(z)}=\beta\frac{z}{2(1+z)^{\frac{3}{2}}}.$$ Further
\begin{eqnarray*}
  \RE\left(\frac{zQ'(z)}{Q(z)}\right) &=& 1-\frac{3}{2}\RE\left(\frac{z}{1+z}\right) \\
   &\geq& \frac{1}{4}>0.
\end{eqnarray*} Thus the function $Q$ is starlike and the result now follows by an application of Lemma~\ref{miller1}.
\end{proof}
In the next result condition on $\beta$ is obtained such that $p(z)+\beta z p'(z)/p(z)\prec \sqrt{1+z}$ implies that $p(z)\prec (1+Az)/(1+Bz).$
\begin{lem} Let $-1\leq B<A\leq 1, (A-B)\beta\geq \sqrt{2}(1+|A|)(1+|B|)+|A|^2-1$ and
 $$\frac{1}{\beta}\geq\max\set{0,\frac{A-B}{(1+|A|)(1+|B|)}-\frac{1-|B|}{1+|B|}}.$$ Let $p$ be an analytic function defined on $\mathbb{D}$ with $p(0)=1$ satisfies
$$p(z)+\beta \frac{zp'(z)}{p(z)}\prec \sqrt{1+z}.$$ Then $p(z)\prec \frac{1+Az}{1+Bz}.$
\end{lem}
\begin{proof} Define the function $q:\mathbb{D}\rightarrow \mathbb{C}$ by $q(z)=(1+Az)/(1+Bz),\; -1\leq B<A\leq 1$. Consider the subordination $$p(z)+\beta \frac{zp'(z)}{p(z)}\prec q(z)+\beta \frac{zq'(z)}{q(z)}.$$ Thus, in view of Lemma \ref{miller2}, the above subordination can be written as (\ref{e1}) by defining the functions $\nu$ and $\varphi$ as $\nu(w):=w$ and $\varphi(w):=\beta/w\; (\beta\neq0).$ Clearly the functions $\nu$ and $\varphi$ are analytic in $\mathbb{C}$ and $\varphi(w)\neq0.$
Let the functions $Q(z)$ and $h(z)$ be defined by $$Q(z):=zq'(z)\varphi(q(z))=\beta \frac{zq'(z)}{q(z)}$$ and
$$h(z):=\nu(q(z))+Q(z)=q(z)+\beta \frac{zq'(z)}{q(z)}.$$
A computation shows that $Q(z)$ is starlike univalent in $\mathbb{D}$. Further
$$\frac{zh'(z)}{Q(z)}=\frac{1}{\beta}+1+\frac{zq''(z)}{q'(z)}-\frac{zq'(z)}{q(z)}.$$
Let $z=e^{it},\; -\pi\leq t\leq \pi$. Then
\begin{eqnarray*}
  \RE\left(\frac{e^{it}h'(e^{it})}{Q(e^{it})}\right) &=& \frac{1}{\beta}+\RE\left(\frac{1-Be^{it}}{1+Be^{it}}-\frac{(A-B)e^{it}}{(1+Ae^{it})(1+Be^{it})}\right)\\
   &=& \frac{1}{\beta}+\frac{1-|B|}{1+|B|}-\frac{A-B}{(1+|A|)(1+|B|)}>0.
\end{eqnarray*}
Thus by Lemma~\ref{miller2}, it follows that $p(z)\prec q(z).$ In order to prove our result, we need to prove that
  $$\Phi(z):=\sqrt{1+z}\prec q(z)+\beta \frac{zq'(z)}{q(z)}=\frac {1+Az}{1+Bz}+\frac{\beta(A-B)z}{(1+Az)(1+Bz)}:=h(z)$$ The subordination $\Phi(z)\prec h(z)$ is equivalent to the subordination  $z\prec \Phi^{-1}(h(z)).$  Now in order to prove result it is enough to show $|\Phi^{-1}(h(e^{it}))|\geq1,\; -\pi\leq t\leq \pi.$ Now
  $$\left|\Phi^{-1}(h(e^{it}))\right|=\left|\left(\frac{1+Ae^{it}}{1+Be^{it}}+\frac{\beta(A-B)e^{it}}
  {(1+Ae^{it})(1+Be^{it})}\right)^2-1\right|\geq1$$ implies
  $$ \left|\frac{1+Ae^{it}}{1+Be^{it}}+\frac{\beta(A-B)e^{it}}
  {(1+Ae^{it})(1+Be^{it})}\right|\geq\sqrt{2}.$$ Further
  \begin{eqnarray*}
    \left|\frac{1+Ae^{it}}{1+Be^{it}}+\frac{\beta(A-B)e^{it}}
  {(1+Ae^{it})(1+Be^{it})}\right| &\geq& \RE\left(\frac{1+Ae^{it}}{1+Be^{it}}+\frac{\beta(A-B)e^{it}}{(1+Ae^{it})(1+Be^{it})}\right) \\
   &\geq& \frac{1-|A|}{1+|B|}+\frac{(A-B)\beta}{(1+|A|)(1+|B|)}\geq\sqrt{2},\end{eqnarray*}
  for $(A-B)\beta\geq\sqrt{2}(1+|A|)(1+|B|)+|A|^2-1$. This completes the proof.
\end{proof}
 \section{Sufficient condition for Janowski Starlikeness}\label{3}
 The following first two results [Lemma~\ref{t}, \ref{t1}] are essentially due to Ali et al. \cite[Lemma 2.1, 2.10]{ali}. However,
an alternate proof of the same is presented below which is much easier than that of given by Ali et al. \cite{ali}:
\begin{lem}\label{t} Assume that $-1\leq B< A\leq1$, $-1\leq E<D\leq1$ and $\beta(A-B)\geq(D-E)(1+B^2)+|2B(D-E)-E\beta(A-B)|$.
Let $p$ be an analytic function defined on $\mathbb{D}$ with $p(0)=1$ satisfies
$$1+\beta zp'(z)\prec \frac{1+Dz}{1+Ez},\;\; \beta\neq0.$$ Then $p(z)\prec \frac{1+Az}{1+Bz}.$
\end{lem}
\begin{proof} Define the function $q:\mathbb{D}\rightarrow \mathbb{C}$ by  $$q(z)= \frac{1+Az}{1+Bz},\;\;-1\leq B< A\leq1.$$ Then $q$ is convex in $\mathbb{D}$ with $q(0)=1$. Further computation shows that
$$Q(z)=\beta zq'(z)=\frac{\beta(A-B)z}{(1+Bz)^2}$$ and $Q$ is starlike in $\mathbb{D}.$  It follows from Lemma \ref{miller}, that the subordination  $$1+\beta zp'(z)\prec1+\beta zq'(z)$$ implies $p(z)\prec q(z).$ In view of the above result it is sufficient to prove
  $$\frac{1+Dz}{1+Ez}\prec 1+\beta zq'(z)=1+ \beta\frac{(A-B)z}{(1+Bz)^2}=h(z).$$
 Let $w=\Phi(z)=\frac{1+Dz}{1+Ez}.$ Then $\Phi^{-1}(w)=\frac{w-1}{D-Ew}$ and
 \begin{eqnarray*}
   \Phi^{-1}(h(z)) &=& \frac{\beta(A-B)z}{D(1+Bz)^2-E(1+Bz)^2-\beta E(A-B)z} \\
    &=& \frac{\beta(A-B)z}{(D-E)(1+B^2z^2)+(2B(D-E)-\beta E(A-B))z}.
 \end{eqnarray*} Let $z=e^{it},\; \pi\leq t\leq \pi$. Thus
 $$|\Phi^{-1}(h(e^{it}))|\geq \frac{|\beta|(A-B)}{(D-E)(1+B^2)+|(2B(D-E)-\beta E(A-B))|}\geq 1,$$ for
 $|\beta|(A-B)\geq (D-E)(1+B^2)+|(2B(D-E)-E\beta(A-B))|.$  Hence $q(\mathbb{D})\subset h(\mathbb{D})$ that is $q(z)\prec h(z)$ this completes the proof.
\end{proof}

It should be noted that Ali et al.\cite{ali} made the assumption $AB>0$ in order to prove the result \cite[Lemma 2.10]{ali}, whereas in the following lemma this condition has been dropped:

\begin{lem}\label{t1} Assume that $-1\leq B< A\leq1$, $-1\leq E<D\leq1$ and $\beta(A-B)\geq(D-E)(1+|AB|)+|(A+B)(D-E)-E\beta(A-B)|$.
Let $p$ be an analytic function defined on $\mathbb{D}$ with $p(0)=1$ satisfies
$$1+\beta \frac{zp'(z)}{p(z)}\prec \frac{1+Dz}{1+Ez},\;\; \beta\neq0.$$ Then $p(z)\prec \frac{1+Az}{1+Bz}.$
\end{lem}
\begin{proof}  As above define the function $q:\mathbb{D}\rightarrow \mathbb{C}$ by  $$q(z)= \frac{1+Az}{1+Bz},\;\;-1\leq B< A\leq1.$$ Then $q$ is convex in $\mathbb{D}$ with $q(0)=1$. A computation shows that $$Q(z)=\frac{\beta zq'(z)}{q(z)}=\frac{\beta(A-B)z}{(1+Az)(1+Bz)}$$ and $Q$ is starlike in $\mathbb{D}.$  It follows from Lemma \ref{miller}, that the subordination  $$1+\beta \frac{zp'(z)}{p(z)}\prec1+\beta \frac{zq'(z)}{q(z)}$$ implies $p(z)\prec q(z).$ Now we need to prove
  $$\frac{1+Dz}{1+Ez}\prec 1+\beta \frac{zq'(z)}{q(z)}=1+ \beta\frac{(A-B)z}{(1+Bz)^2}=h(z).$$
 Let $w=\Phi(z)=\frac{1+Dz}{1+Ez}.$ Then $\Phi^{-1}(w)=\frac{w-1}{D-Ew}$ and
 \begin{eqnarray*}
   \Phi^{-1}(h(z)) &=& \frac{\beta(A-B)z}{(D-E)(1+Az)(1+Bz)-\beta E(A-B)z} \\
    &=& \frac{\beta(A-B)z}{(D-E)(1+ABz^2)+((A+B)(D-E)-\beta E(A-B))z}.
 \end{eqnarray*} Let $z=e^{it},\; \pi\leq t\leq \pi$. Thus
 $$|\Phi^{-1}(h(e^{it}))|\geq \frac{|\beta|(A-B)}{(D-E)(1+|AB|)+|(A+B)(D-E)-\beta E(A-B)|}\geq 1,$$ for
 $|\beta|(A-B)\geq (D-E)(1+|AB|)+|(A+B)(D-E)-E\beta(A-B)|.$  Hence $q(\mathbb{D})\subset h(\mathbb{D})$ that is $q(z)\prec h(z)$ this completes the proof.
\end{proof}
\begin{lem}\label{th2} Assume that $-1\leq B< A\leq1$, $-1\leq E<D\leq1$ and $|\beta|(A-B)\geq(D-E)(1+A^2)+|2A(D-E)-E\beta(A-B)|$.
Let $p$ be an analytic function defined on $\mathbb{D}$ with $p(0)=1$ satisfies
$$1+\beta \frac{zp'(z)}{p^2(z)}\prec \frac{1+Dz}{1+Ez}.$$ Then $p(z)\prec \frac{1+Az}{1+Bz}.$
\end{lem}
\begin{proof} Define the function $q:\mathbb{D}\rightarrow \mathbb{C}$ by  $$q(z)= \frac{1+Az}{1+Bz},\;\;-1\leq B< A\leq1.$$ Then $q$ is convex in $\mathbb{D}$ with $q(0)=1$. A computation shows that
$$Q(z)=\frac{\beta zq'(z)}{q^2(z)}=\frac{\beta(A-B)z}{(1+Az)^2}$$ and $$\frac{zQ'(z)}{Q(z)}=\frac{1-Az}{1+Az}.$$
 As before, a computation shows $Q$ is starlike in $\mathbb{D}.$  It follows from Lemma \ref{miller}, that the subordination  $$1+\beta \frac{zp'(z)}{p^2(z)}\prec1+\beta \frac{zq'(z)}{q^2(z)}$$ implies $p(z)\prec q(z).$ To prove result, it is enough to show that
  $$\frac{1+Dz}{1+Ez}\prec 1+\beta \frac{zq'(z)}{q^2(z)}=1+ \beta\frac{(A-B)z}{(1+Az)^2}=h(z).$$
Remaining part of the proof is similar to that of Lemma~\ref{t} and therefore it is skipped here.
\end{proof}
\begin{rem}
 When $\beta=1$, Lemma~\ref{th2} reduces to \cite[Lemma 2.6]{ali} due to Ali et al.
\end{rem}
\bibliographystyle{amsplain}

{\bf Acknowledgments.}  The research work is supported by a grant from University of Delhi and  also by the Basic Science Research Program through the National Research Foundation of Korea(NRF) funded by the Ministry of Education, Science and Technology (No. 2012-0002619).

\end{document}